\documentclass{amsart}
 \newtheorem{thm}{Theorem}[section]
 
 \newtheorem{lem}[thm]{Lemma}
 \newtheorem{prop}[thm]{Proposition}
 \theoremstyle{definition}
 
 \theoremstyle{remark}

 \numberwithin{equation}{section}
 
 \makeatletter
\@namedef{subjclassname@2010}{
  \textup{2010} Mathematics Subject Classification}
\makeatother

\begin{document}

%
%
%
%
%
%
%
%
%

\title[Global semianalytic sets defined by definable functions]
 {Closure and Connected Component of a Planar Global Semianalytic Set Defined by Analytic Functions Definable in O-minimal Structure}

\author[M. Fujita]{Masato Fujita}

\address{%
Department of Liberal Arts,
Japan Coast Guard Academy,
5-1 Wakaba-cho, Kure, Hiroshima 737-8512, Japan}

\email{masato@jcga.ac.jp}



\date{}

\begin{abstract}
We consider a global semianalytic set defined by real analytic functions definable in an o-minimal structure. When the o-minimal structure is polynomially bounded, we show that the closure of this set is a global semianalytic set defined by definable real analytic functions. We also demonstrate that a connected component of a planar global semianalytic set defined by real analytic functions definable in a substructure of the restricted analytic field is a global semianalytic set defined by definable real analytic functions.
\end{abstract}

\maketitle

\section{Introduction}\label{intro}

A semianalytic set is defined as a subset of a real analytic manifold that is, at each point of the manifold, a finite union of sets defined by finite inequalities and equalities of real analytic functions defined on a neighborhood of the point. \L ojasiewicz \cite{Lojasiewicz} proved that the closure and a connected component of a semianalytic set are also semianalytic. A subset of a real analytic manifold is called global semianalytic if it is a finite union of sets defined by finite inequalities and equalities of real analytic functions defined on the whole manifold. It is not yet known whether a connected component of a global semianalytic set is itself global semianalytic, except in some specific cases \cite{AC, F}. 

We want to consider this problem in the o-minimal setting. Consider an o-minimal expansion $\Tilde{\mathbb R}$ of the real field. (See \cite{vdD} for the theory of o-minimal structures.) A global semianalytic set defined by real analytic functions definable in the o-minimal structure $\Tilde{\mathbb R}$ takes the form:
\begin{center}
$\displaystyle \bigcup_{i=1}^k \{x \in \mathbb R^n\;|\;f_i(x)=0,g_{i1}(x)>0, \ldots, g_{il}(x)>0 \}$, 
\end{center}
where $f_i$ and $g_{ij}$ are definable real analytic functions on $\mathbb R^n$. In this paper, we refer to such sets as being \textit{globally $\Tilde{\mathbb R}$-definable semianalytic} or \textit{globally definable semianalytic}. A global semianalytic set that is simultaneously definable in an o-minimal structure is not necessarily a globally definable semianalytic set, an example of which is shown below.

Consider the restricted analytic field $\mathbb R_{\text{an}}$. A subset of $\mathbb R^n$ is definable in $\mathbb R_{\text{an}}$ if and only if its closure in the $n$-dimensional projective space $P^n(\mathbb R)$ is subanalytic. (See \cite{BM} for the definition of a subanalytic set.) Here, we identify $\mathbb R^n$ with an open subset of $P^n(\mathbb R)$ under the identification given by 
\begin{center}
$\mathbb R^n \ni (x_1, \ldots, x_n) \mapsto (1:x_1: \ldots:x_n) \in P^n(\mathbb R)$. 
\end{center}
Set 
$X=\{(t,\sin(t))\;|\;0 \leq t \leq \pi\}$.
Then, $X$ is a compact global semianalytic set; furthermore, $X$ is definable in $\mathbb R_{\text{an}}$. However, its analytic closure is the sine curve that intersects with the real line $\{(t,0); t \in \mathbb R\}$ at infinitely many points. The analytic closure of a globally definable semianalytic set intersects with the real line at a finite number of points. The set $X$ is simultaneously definable in $\mathbb R_{\text{an}}$ and global semianalytic, but it is not a globally definable semianalytic set.

In this paper, o-minimal structure indicates an o-minimal expansion of the real field. We show that the closure of a globally definable semianalytic set is a globally definable semianalytic set in the case where the o-minimal structure is polynomially bounded. (See \cite{Miller2, vdDM3} for the definition of polynomially bounded o-minimal structures.) We also show that a connected component of a planar globally definable semianalytic set is again a globally definable semianalytic set when the o-minimal structure is a substructure of the restricted analytic field.

\section{Closure and Finiteness Theorem}\label{cf}
In this section, we demonstrate that the closure of a globally definable semianalytic set is a globally definable semianalytic set when the o-minimal structure is polynomially bounded with a proof based on Gabrielov's work \cite[Lemma 1]{Gabrielov}.
\begin{thm}\label{closures}
Consider a polynomially bounded o-minimal structure. Let $X$ be a globally definable semianalytic subset of $\mathbb R^n$. Then, the closure $\overline{X}$ is again a globally definable semianalytic set.
\end{thm}
\begin{proof}
We may assume that 
\begin{center}
$X=\{x \in \mathbb R^n\;|\;f(x)=0, g_1(x)>0, \ldots, g_k(x)>0\}$ 
\end{center}
without loss of generality. Here, $f(x)$ and $g_1(x), \ldots, g_k(x)$ are definable real analytic functions on $\mathbb R^n$.
\medskip

There exists a positive integer $q>0$ such that the closure of the set 
\begin{center}
$Y_{x,c}=\{y \in X\;|\;g_1(y) > c\|x-y\|^q, \ldots,g_k(y) > c\|x-y\|^q  \} \subset \mathbb R^n$
\end{center}
contains $x$ if and only if $x \in \overline{X}$ for any positive real number $c$. Here, $\|x-y\|$ denotes the Euclidean norm of element $x-y$ in $\mathbb R^n$. Obviously, for any positive integer $q$, the condition that $x \in X$ implies that $x \in Y_{x,c}$; likewise, $x \not\in \overline{Y_{x,c}}$ if $x \not\in \overline{X}$ because $Y_{x,c} \subset X$.

Finally, we consider the case in which $x \in \partial X$, where $\partial X$ denotes the boundary of $X$ defined by $\partial X = \overline{X} \setminus X$. Consider the positive definable function defined by 
\begin{center}
$\delta(x,t)=\displaystyle\max_{y\in X, \|x-y\|=t}\left(\min_{1 \leq i \leq k} g_i(y) \right)$
\end{center}
for any $x \in \partial X$ and any sufficiently small $t>0$. We show that the function $\delta(x,t)$ is well-defined and positive. Set $g(y)=\displaystyle\min_{1 \leq i \leq k} g_i(y)$, which is a continuous definable function on $\mathbb R^n$ and positive on $X$. Set $X_{x,t}=\{y \in X\;|\; \|x-y\|=t\}$ for $x \in \partial X$ and a sufficiently small $t>0$; then, $g(y)=0$ for any $y \in \overline{X_{x,t}} \setminus X_{x,t}$. Since $\overline{X_{x,t}}$ is compact, the restriction of $g$ to $\overline{X_{x,t}}$ has the maximum value, which is positive because $g$ is positive on $X_{x,t}$. Hence, the restriction of $g$ to $\overline{X_{x,t}}$ has the positive maximum value at a point in $X_{x,t}$. We have shown that the function $\delta(x,t)$ is well-defined and positive for any $x \in \partial X$ and any sufficiently small $t>0$.

There exists $q \in \mathbb N$ with $\delta(x,t) > ct^q$ for any positive real number $c$ and any sufficiently small $t>0$ by \cite[Theorem 1.4]{vdDM3}. Thus, there exists $y_t \in Y_{x,c}$ with $\|x-y_t\|=t$ for any sufficiently small positive real number $t$. Therefore, the closure of $Y_{x,c}$ contains $x$ if $x \in \partial X$.
\medskip

Let $\Hat{g}_i:\mathbb R^n \times \mathbb R^n \rightarrow \mathbb R$ be the definable function such that $\Hat{g}_i(x,y)$ is the Taylor expansion of $g_i(y)$ of order $q$ at $x$. That is, we set 
\begin{center}
$\Hat{g}_i(x,y)= \displaystyle \sum_{\mu \in (\mathbb N \cup \{0\})^n, |\mu| \leq q} \dfrac{\partial^{|\mu|}g_i(x)}{\partial^{\mu_1}X_1 \cdots \partial^{\mu_n}X_n} 
(y_1-x_1)^{\mu_1} \cdots (y_n-x_n)^{\mu_n}$
\end{center}
for $\mu=(\mu_1, \ldots, \mu_n)$, $x=(x_1, \ldots, x_n) \in \mathbb R^n$ and $y=(y_1, \ldots, y_n) \in \mathbb R^n$. Here, $|\mu|=\sum_{i=1}^n \mu_i$ and $X_1, \ldots, X_n$ are the coordinate functions of $\mathbb R^n$. Set 
\begin{center}
$S_{x,c}=\{y \in \mathbb R^n\;|\; \Hat{g}_i(x,y) > 0 \text{ and }\Hat{g}_i(x,y) \geq c\|y-x\|^q \text{ for all }i=1 ,\ldots, k \}$. 
\end{center}
We show that the closure of the set 
\begin{center}
$S'_{x,c}:=S_{x,c} \cap \{y \in \mathbb R^n\;|\; f(y)=0\}$ 
\end{center}
contains $x$ if and only if $x \in \overline{X}$. 

Assume first that $x \in \overline{X}$. Therefore, $x \in \overline{Y_{x,2c}}$ and there exists a sequence $\{y_\nu\} \subset Y_{x,2c}$ with $y_\nu \to x$. Since $\Hat{g}_i(x,y)$ is the Taylor expansion of $g_i(y)$ of the order $q$ at $x$, there exists $C>0$ with $|g_i(x)-\Hat{g}_i(x,y)| \leq C\|x-y\|^{q+1}$ in a neighborhood of $x$. If $y$ is sufficiently close to $x$, $c\|x-y\|^q > C\|x-y\|^{q+1}$. Hence, we have $ \Hat{g}_i(x,y) \geq c\|y-x\|^q$ if $y \in Y_{x,2c}$ and $y$ is sufficiently close to $x$, and $y_\nu \in S'_{x,c}$ for sufficiently large $\nu>0$. Therefore, we have $x \in \overline{S'_{x,c}}$. We next consider the case in which $x \in \overline{S'_{x,c}}$. We can show that $x \in  \overline{Y_{x,\frac{1}{2}c}}$ in the same way as above. Hence, $x \in \overline{X}$.
\medskip

We next demonstrate that there exists a positive integer $q>0$ such that the closure of the set 
\begin{center}
$Y'_{x,c,c'}=\{y \in S_{x,c}\;|\;|f(y)| \leq c'\|x-y\|^{q'} \}$
\end{center}
contains $x$ if and only if $x \in \overline{X}$ for any positive real number $c'$. Note that $S'_{x,c} \subset Y'_{x,c,c'}$, so the closure of $Y'_{x,c,c'}$ contains $x$ if $x \in \overline{X}$. We next show that $x \not\in \overline{Y'_{x,c,c'}}$ if $x \not\in \overline{X}$, a claim that is obvious in the case in which $f(x) \not=0$. We assume that $f(x)=0$. If $x \not\in \overline{S_{x,c}}$, then we have $x \not\in \overline{S'_{x,c}}$, $x \not\in \overline{X}$, and $x \not\in \overline{Y'_{x,c,c'}}$, in which case the claim is also true. It is also obvious that the closure of $Y'_{x,c,c'}$ does not contain $x$ if $g_i(x)<0$ for some $1 \leq i \leq k$. Therefore, we have only to consider the case in which $x \in \overline{S_{x,c}}$ and $f(x)= \displaystyle\min_{1 \leq i \leq k} g_i(x)=0$. 

Set $V=\{x \in \mathbb R^n\;|\;f(x)= \displaystyle\min_{1 \leq i \leq k} g_i(x)=0, x \in \overline{S_{x,c}} \} \setminus \overline{X}$. $V$ is a definable set. Consider the definable function defined by  
\begin{center}
$\delta'(x,t)=\displaystyle  \min_{y \in S_{x,c}, \|y-x\| = t} |f(y)|$ 
\end{center}
for any $x \in V$ and sufficiently small $t>0$. Set $S_{x,c,t} = \{y \in S_{x,c}\;|\;\|y-x\|=t\} = \{y \in \mathbb R^n\;|\;\|y-x\|=t, \Hat{g}_i(x,y) \geq c\|y-x\|^q\ (\forall i)\}$, which is a compact set. Because $|f(y)|$ is continuous, the restriction of $|f(y)|$ to $S_{x,c,t}$ has the minimum value. We have shown that function $\delta'(x,t)$ is well-defined. We next establish that $\delta'(x,t)$ is positive for any $x \in V$ and sufficiently small number $t>0$ by assuming the contrary. There exists $x \in V$ such that $\delta'(x,t)=0$ for any sufficiently small $t>0$. Then, there also exists a sequence $\{y_\nu\} \subset S_{x,c}$ with $y_\nu \to x$ and $f(y_\nu)=0$, which means that $y_\nu \in S'_{x,c}$. We now have $x \in \overline{S'_{x,c}}$ and $x \in \overline{X}$, which contradicts the assumption that $x \not\in \overline{X}$.

We can choose a positive integer $q'>0$ with $\delta'(x,t) > c' t^{q'}$ for any $x \in V$, any positive real number $c'$, and any sufficiently small $t>0$ by \cite[Theorem 1.4]{vdDM3}. Therefore, the closure of $Y'_{x,c,c'}$ does not contain $x$ if $x \in V$. 
\medskip

Let $\Hat{f}:\mathbb R^n \times \mathbb R^n \rightarrow \mathbb R$ be the definable function such that $\Hat{f}(x,y)$ is the Taylor expansion of $f(y)$ of order $q'$ at $x$. We can show that the closure of set 
\begin{center}
$T_{x,c,c'}=\{y \in S_{x,c}\;|\;|\Hat{f}(x,y)| \leq c' \|y-x\|^{q'}\}$
\end{center}
contains $x$ if and only if $x \in \overline{X}$, as in the case of $S'_{x,c}$. (We omit this proof.)
\medskip

For a fixed $x$, $T_{x,c,c'}$ is a semialgebraic set. According to the quantifier elimination theorem \cite[Proposition 5.2.2]{BCR}, \cite{Tarski}, the condition that the closure of $T_{x,c,c'}$ contains $x$ is equivalent to a semialgebraic condition on coefficients of polynomials in the variable $y$ defining $T_{x,c,c'}$. The coefficients are the partial derivatives of $f$ and $g_i$, which are definable and real analytic. Hence, $\overline{X}$ is a globally definable semianalytic set.
\end{proof}

In semialgebraic geometry, a closed (resp. open) set is a finite union of basic closed (resp. open) sets \cite[Theorem 2.7.2]{BCR}. A similar finiteness theorem follows:
\begin{prop}[Finiteness Theorem]\label{finite}
Consider a polynomially bounded o-minimal structure. An open {\rm{(}}resp. closed{\rm{)}} globally definable semianalytic set $S$ takes the form 
\begin{center}
$S= \displaystyle\bigcup_{i=1}^k \bigcap_{j=1}^l \{x \in \mathbb R^n\;|\; f_{ij}(x)> \text{{\rm{(}}resp.} \geq \text{{\rm{)}} } 0\}$
\end{center}
Here, $f_{ij}(x)$ are definable real analytic functions on $\mathbb R^n$. 
\end{prop}
\begin{proof}
We only prove the case in which $S$ is open because by taking the complement, the closed case is clear. \par 
Set $S$ is a finite union of global semianalytic sets of the form
\begin{center}
$S'=\{x \in \mathbb R^n\;|\;f(x)=0, g_1(x)>0, \ldots, g_m(x)>0\}$.
\end{center}
Here, $f(x)$ and $g_1(x),\ldots,g_m(x)$ are definable real analytic functions on $\mathbb R^n$. Let $\|x\|$ denote the Euclidean norm of $x \in \mathbb R^n$. We may assume that $|f(x)| \leq \frac{1}{\sqrt{1+\|x\|^2}}$ for any $x \in \mathbb R^n$ by replacing $f(x)$ with the bounded function $\frac{f(x)}{\sqrt{(1+f(x)^2)(1+\|x\|^2)}}$ if necessary. We may also assume that $|g_i(x)| \leq \frac{1}{\sqrt{1+\|x\|^2}}$ for all $1 \leq i \leq m$.

Let $B$ be the open unit ball in $\mathbb R^n$ centered at the origin, i.e., $B=\{x \in \mathbb R^n\;|\; \|x\| < 1\}$. Let $D$ be the closed unit ball in $\mathbb R^n$ centered at the origin, i.e., $D=\{x \in \mathbb R^n\;|\; \|x\| \leq 1\}$. Consider the Nash diffeomorphism $\varphi: \mathbb R^n \rightarrow B$ given by $\varphi(x) = \left( \frac{x_1}{\sqrt{1+\|x\|^2}}, \ldots, \frac{x_n}{\sqrt{1+\|x\|^2}}\right)$ for $x=(x_1, \ldots, x_n)$. Note that $\varphi(S)$ is open in $\mathbb R^n$. The definable functions $F: D \rightarrow \mathbb R$ and $G_i: D \rightarrow \mathbb R$ are defined as follows:
\begin{center}
$F(u)=\left\{
\begin{array}{cc}
f\circ \varphi^{-1}(u) & \text{if } u \in B,\\
0 & \text{otherwise}
\end{array}
\right.
$
$G_i(u)=\left\{
\begin{array}{cc}
g_i\circ \varphi^{-1}(u) & \text{if } u \in B,\\
0 & \text{otherwise.}
\end{array}
\right.
$
\end{center}
We can easily show that $|F(u)| \leq \sqrt{1-\|u\|^2}$ and $|G_i(u)| \leq \sqrt{1-\|u\|^2}$ for any $u \in B$. Hence, functions $F$ and $G_i$ are continuous on $D$. 

Set $g(x)= \prod_{i=1}^m(|g_i(x)|+g_i(x))$; in this case, $g(x)=0$ if $f(x)=0$ and $x \not\in S$. Set $G(u)= \prod_{i=1}^m(|G_i(u)|+G_i(u))$. For the compact set $D \setminus \varphi(S)$, we have $G(u)=0$ if $F(u)=0$. For some positive integer $N$ and $\widetilde{c}>0$, by the \L ojasiewicz inequality \cite[5.4]{Miller2}, $G^N \leq \widetilde{c} |F|$ on $D \setminus \varphi(S)$. Hence, $g^N \leq \widetilde{c} |f|$ on $\mathbb R^n \setminus S$. Then, set $S''=\{\widetilde{c}^2 f^2<(2^m \prod_{i=1}^m g_i)^{2N},g_1>0 ,\ldots, g_m>0\}$ is contained in $S$ and contains $S'$. By replacing $S'$ with $S''$, we can show the proposition inductively. 
\end{proof}

\section{Separation}\label{separation}
In this section, we show that a connected component of a planar globally definable semianalytic set is again a globally definable semianalytic set in the case in which the o-minimal structure is a substructure of the restricted analytic field. We begin by illustrating several lemmas.

A \textit{global analytic subset} of a paracompact Hausdorff real analytic manifold is the common zero set of finite real analytic functions defined on the manifold. A global analytic set $X$ is \textit{irreducible} if it is not the union of two proper global analytic subsets of $X$. (See \cite[Chapter VIII]{ABR} for global analytic sets.)

\begin{lem}\label{simple3}
Consider an o-minimal structure $\Tilde{\mathbb R}$ that admits analytic decomposition. (See \cite{vdDM2} for the definition of analytic decomposition.) Let $X$ be a proper global analytic subset of $\mathbb R^2$ that is definable in $\Tilde{\mathbb R}$. Then, one-dimensional irreducible global analytic subsets of $X$ are definable, and at most a finite number of them exist.
\end{lem}
\begin{proof}
Because $\Tilde{\mathbb R}$ admits analytic decomposition, set $X$ is the union of finite cells definable in $\Tilde{\mathbb R}$ that are analytically diffeomorphic to Euclidean spaces. These cells satisfy the following conditions: First, the intersection of two cells is empty. Second, the boundary of a cell is a union of cells of smaller dimension. Furthermore, they are real analytic submanifolds of $\mathbb R^2$ because they are either points, or sets of the forms
\begin{center}
  $\{(x,y) \in \mathbb R^2\;|\;x=a,y\in I\}$ and
    $\{(x,y) \in \mathbb R^2\;|\;y=\xi(x),x\in I\}$, 
 \end{center} 
where $a$ is a constant, $I$ is an open interval in $\mathbb R$, and $\xi(x)$ is a definable real analytic function defined on $I$. Hence, a one-dimensional cell is contained in the set of regular points of $X$.

If two one-dimensional irreducible global analytic subsets of $X$ intersect at point $z$, the ring of analytic function germs on $X$ at $z$ is not an integral domain. By \cite[Theorem 20.3]{Mat}, it is also not a regular local ring. Hence, $z$ is a singular point of $X$. A cell is contained in a single irreducible global analytic subset of $X$. Hence, a one-dimensional irreducible global analytic subset of $X$ is a union of cells and is definable. Furthermore, the number of one-dimensional irreducible global analytic subsets of $X$ is finite.
\end{proof}

\begin{lem}\label{simple}
Consider closed semianalytic set germs $A$ and $B$ of $\mathbb R^n$ at the origin $O$ such that the intersection $A \cap B$ is the origin. There exists an open semialgebraic subset $C$ of $\mathbb R^n$ satisfying the conditions that $A \subset C \cup \{O\}$ and $B \cap \overline{C} = \{O\}$ as set germs at the origin.
\end{lem}
\begin{proof}
There exists an open semianalytic set germ $C'$ at the origin with $A \subset C' \cup \{O\}$ and $B \cap \overline{C'} = \{O\}$ by \cite[Lemma 2.4]{F}. Let $C'=\bigcup_{i=1}^l \bigcap_{j=1}^m \{ f_{ij} > 0\}$, where $f_{ij}(x)$ are real analytic functions defined on a neighborhood of the origin. By \cite[Lemma 2.5]{F}, there exists a positive integer $\mu$ such that any semianalytic set germ $C=\bigcup_{i=1}^l \bigcap_{j=1}^m \{ g_{ij} > 0\}$ satisfies the conditions that $A  \subset C \cup \{O\}$ and $B \cap \overline{C} = \{O\}$ as set germs at the origin if $f_{ij}-g_{ij} \in \mathfrak{m}^\mu$, where $\mathfrak{m}$ is the maximal ideal of the ring of real analytic function germs at the origin. We can choose $g_{ij}$ as polynomials and hence, we can choose $C$ as a semialgebraic set.
\end{proof}

\begin{lem}\label{lem1}
Consider an o-minimal substructure $\Tilde{\mathbb R}$ of the restricted analytic field $\mathbb R_{\text{an}}$. Let $C$ be a globally $\Tilde{\mathbb R}$-definable semianalytic subset of $\mathbb R^2$. Let $X$ and $Y$ be $\Tilde{\mathbb R}$-definable subsets of $\mathbb R^2$ with $X \cup Y = C$ and $X \cap Y = \emptyset$. Assume further that there exist a non-zero $\Tilde{\mathbb R}$-definable real analytic function $g$ on $\mathbb R^2$, a finite subset $P$ of $\mathbb R^2$, and finitely many one-dimensional $\mathbb R_{\text{an}}$-definable irreducible global analytic subsets $Z_1,\ldots, Z_m$ of $\mathbb R^2$ such that $\overline{X} \cap \overline{Y} \subset P \cup \left(\bigcup_{i=1}^m Z_i\right)$ and the function $g$ vanishes on $Z_i$ for all $i=1, \ldots, m$. Then $X$ and $Y$ are both globally $\Tilde{\mathbb R}$-definable semianalytic.
\end{lem}
\begin{proof}
We simply call a $\Tilde{\mathbb R}$-definable set \textit{definable} in this proof-by-induction on $m$. Set $W=\overline{X} \cap \overline{Y}$.

We first consider the case in which $m=0$. In this case, $W$ is either empty or consists of a finite number of points. Let $q$ be the cardinality of $W$.
We show the lemma by induction on $q$. For the case in which $q=0$, $W$ is empty. Consider the Nash map $\varphi: \mathbb R^2 \rightarrow S^2$, defined by $\varphi(x) = \left(\frac{2x_1}{{1+\|x\|^2}}, \frac{2x_2}{{1+\|x\|^2}}, \frac{1-\|x\|^2}{{1+\|x\|^2}}\right)$. Here, $S^2$ is the sphere $\{(x_1,x_2,x_3) \in \mathbb R^3\;|\; x_1^2+x_2^2+x_3^2=1\}$ in $\mathbb R^3$ and $x=(x_1, x_2)$. It is a Nash diffeomorphism to the image $S^2\setminus \{S\}$, where $S=(0,0,-1)$. Set $X_1=\varphi(X)$ and $Y_1=\varphi(Y)$. Thus, $W'=\overline{X_1} \cap \overline{Y_1}$ is either empty or consists of a single point $S$.

In the case in which $W'$ is empty, let $d_{X_1}, d_{Y_1}:\mathbb R^3 \rightarrow \mathbb R$ be the distance functions to the closures $\overline{X_1}$ and $\overline{Y_1}$, respectively. Define a continuous definable function $d:\mathbb R^3 \rightarrow \mathbb R$ by $d(x)=\dfrac{d_{Y_1}(x)-d_{X_1}(x)}{d_{X_1}(x)+d_{Y_1}(x)}$; then, the restriction of $d$ to $X_1$ is $1$ and the restriction of $d$ to $Y_1$ is $-1$. The restriction of $d$ to $S^2$ is a continuous function defined on the compact set $S^2$. There exists a polynomial function $P: \mathbb R^3 \rightarrow \mathbb R$ satisfying the condition that $|P(x)-d(x)| < \frac{1}{2}$ on $S^2$ by the Stone-Weierstrass theorem. We now have $C \cap \{x \in \mathbb R^2\;|\;P(\varphi(x))>0\}=X$. Hence, $X$ is a globally definable semianalytic set. We can also show that $Y$ is a globally definable semianalytic set.

We next consider the case in which $W'$ consists of a single point. Since the o-minimal structure is a substructure of the restricted analytic field $\mathbb R_{\text{an}}$, sets $\overline{X_1}$ and $\overline{Y_1}$ are both definable in $\mathbb R_{\text{an}}$. It is subanalytic by the definition of the restricted analytic field $\mathbb R_{\text{an}}$. Hence it is semianalytic by \cite[Theorem 6.1]{BM}. Therefore, there exists an open semialgebraic subset $V_1$ of $\mathbb R^3$ such that $X_1 \subset V_1$ and $Y_1 \cap \overline{V_1} = \emptyset$ in a neighborhood of the point $S$ by Lemma \ref{simple}. Set $V = \varphi^{-1}(V_1)$, then it is a semialgebraic set. There exists a positive real number $R$ such that $C \cap \{x \in \mathbb R^2\;|\; \|x\|>R\} \cap V = X \cap \{x \in \mathbb R^2\;|\; \|x\|>R\}$. Hence $X \cap \{x \in \mathbb R^2\;|\; \|x\|>R\}$ is a globally definable semianalytic set. Applying the lemma for the case where $W'$ is empty to the set $C \cap \{x \in \mathbb R^2\;|\; \|x\|\leq R\}$, we can show that $X \cap \{x \in \mathbb R^2\;|\; \|x\|\leq R\}$ is also globally definable semianalytic. We have shown that $X=(X \cap \{\|x\|\leq R\}) \cup (X \cap \{\|x\|> R\})$ is globally definable semianalytic. We have finished the proof of the lemma for $m=0$ and $q=0$.
\medskip

We next consider the case in which $m=0$ and $q>0$. Let $p$ be a point in $W$. Remark that $X$ and $Y$ are semianalytic sets because a connected component of a seminalaytic set is semianalytic by \cite[Corollary 2.7, Corollary 2.8]{BM}. We can construct an open semialgebraic set $U$ with $X \subset U$ and $Y \cap \overline{U} = \emptyset$ in a neighborhood of $p$ by Lemma \ref{simple}. The cardinality of the set $\overline{X \cap U} \cap \overline{Y \cap U}$ is smaller than that of $W$. The sets $X \cap U$ and $Y \cap U$ are globally definable semianalytic set by induction, as are $X \setminus U$ and $Y \setminus U$. Hence, $X = (X \cap U) \cup (X \setminus U)$ and set $Y$ are globally definable semianalytic sets, finishing the proof of the lemma for $m=0$.
\medskip

We finally consider the case in which $m>0$. Let $z$ be an arbitrary regular point of $Z_1$. Let $\mathcal O_{z}$ be the ring of real analytic function germs on $\mathbb R^2$ at  $z$. Let $\mathcal I_{Z_1,z}$ be the prime ideal of $\mathcal O_{z}$ of the germs of real analytic functions  vanishing on $Z_1$ at $z$ and is a principal ideal by \cite[Theorem 20.1, Theorem 20.3]{Mat}. Assume that there exists an $\Tilde{\mathbb R}$-definable (resp. $\mathbb R_{\text{an}}$-definable) real analytic function $\psi$ on $\mathbb R^2$ with $\psi \in \mathcal I_{Z_1,z}^k$ for some positive integer $k$. Then, we can easily demonstrate that there exists an $\Tilde{\mathbb R}$-definable (resp. $\mathbb R_{\text{an}}$-definable) real analytic function $\psi'$ on $\mathbb R^2$ with $\psi' \in \mathcal I_{Z_1,z}^{k-1}$. In fact, let $x_1,x_2$ be the coordinate functions of $\mathbb R^2$. There exist real numbers $a$ and $b$ such that vector $(a,b)$ is neither zero nor parallel to the tangent line $T_{z}Z_1$ of $Z_1$ at $z$.
Set $\psi'=a  \dfrac{\partial \psi}{\partial x_1} + b \dfrac{\partial \psi}{\partial x_2}$. Thus, $\psi' \in \mathcal I_{Z_1,z}^{k-1}$. 

Fix a regular point $z_1$ of $Z_1$. Let $A$ be the ring of real analytic functions on $\mathbb R^2$ definable in $\mathbb R_{\text{an}}$. Let $\mathfrak{m}$ be the maximal ideal of $A$ given by $\mathfrak{m}=\{f\in A\;|\; f(z_1)=0\}$. The ring $A_{\mathfrak{m}}$ is a two-dimensional regular local ring by \cite[Remark 2, Proposition 5]{FS} and its proof. Let $\mathfrak{p}$ be the prime ideal of $A$ given by $\mathfrak{p}=\{f \in A\;|\; f(x)=0 \ \text{ for all }x \in Z_1\}$. Then, $I=\mathfrak{p}A_{\mathfrak{m}}$ is a principal ideal by \cite[Theorem 20.1, Theorem 20.3]{Mat}. Let $h_1 \in \mathfrak{p}$ be the generator of $I$. As a corollary of the above claim, we have $h_1 \in \mathcal I_{Z_1,z_1} \setminus \mathcal I_{Z_1,z_1}^2$. There also exists an $\Tilde{\mathbb R}$-definable real analytic function $f_1 \in \mathfrak{p}$ with $f_1 \in \mathcal I_{Z_1,z_1} \setminus \mathcal I_{Z_1,z_1}^2$ by the same claim, because the $\Tilde{\mathbb R}$-definable function $g$ is in $\mathcal I_{Z_1,z_1}$. Since $h_1$ is a generator of $I$, there exist $v \in A$ and $w \in A \setminus \mathfrak{m}$ with $wf_1 = vh_1$. We have $v \not\in \mathfrak{p}$ and $h_1,f_1 \in \mathfrak{p}\setminus\mathfrak{p}^2$ because $f_1, h_1 \in \mathcal I_{Z_1,z_1} \setminus \mathcal I_{Z_1,z_1}^2$.

Set $W_1 = \overline{\{f_1>0\} \cap X} \cap \overline{\{f_1>0\} \cap Y}$. Consider an arbitrary regular point $z \in Z_1$ with $w(z) \not=0$. Let $h_z$ be the generator of $\mathcal I_{Z_1,z}$; then, $f_1=v_zh_z$ for some $v_z \in \mathcal O_{z} \setminus \mathcal I_{Z_1,z}$ because $v \not\in \mathfrak{p}$, $h_1 \in \mathfrak{p}\setminus\mathfrak{p}^2$ and $Z_1$ is an irreducible analytic set. If $v_z(z)\not=0$, then $W_1$ is empty in a neighborhood of $z$. Set $T=Z_1 \cap \left( v^{-1}(0) \cup \overline{h_1^{-1}(0) \setminus Z_1} \cup w^{-1}(0) \right) \cup \operatorname{Sing}(Z_1)$, where $\operatorname{Sing}(Z_1)$ denotes the singular locus of $Z_1$. We can easily show that $\operatorname{Sing}(Z_1)$ is a finite set as in Lemma \ref{simple3}. Thus, we have $W_1 \cap Z_1 \subset T$. The set $T$ is an $\mathbb R_{\text{an}}$-definable set of dimension smaller than one. Hence, there exists a finite subset $P' \subset \mathbb R^2$ with $W_1 \subset P' \cup \left(\bigcup_{i=2}^m Z_i\right)$. Both $\{f_1>0\} \cap X$ and $\{f_1>0\} \cap Y$ are globally definable semianalytic sets by induction. Using the same reasoning, $\{f_1<0\} \cap X$ and $\{f_1<0\} \cap Y$ are globally definable semianalytic sets. By the lemma for $m=0$, $\{f_1=0\} \cap X$ and $\{f_1=0\} \cap Y$ are also globally definable semianalytic sets. We have therefore shown that $X$ and $Y$ are globally definable semianalytic sets.
\end{proof}

The following theorem is the main theorem in this section.

\begin{thm}\label{connected}
Consider an o-minimal substructure $\Tilde{\mathbb R}$ of the restricted analytic field. A connected component of a globally definable semianalytic subset of $\mathbb R^2$ is also a globally definable semianalytic set.
\end{thm}
\begin{proof}
Let $C$ be a globally definable semianalytic subset of $\mathbb R^2$. We may assume that $C$ is of the form
\begin{center}
$ \{x \in \mathbb R^n\;|\;f(x)=0,g_{1}(x)>0, \ldots, g_{l}(x)>0 \}$, 
\end{center}
where $f$ and $g_{j}$ are definable real analytic functions on $\mathbb R^n$ for any $1 \leq j \leq l$. Let $X$ be a connected component of $C$. Set $Y=C \setminus X$. Both $X$ and $Y$ are definable because a connected component is definable by \cite[Proposition 2.18]{vdD}. We show that $X$ is a globally definable semianalytic set. Set $W=\overline{X} \cap \overline{Y}$. The theorem follows from Lemma \ref{lem1} if $W$ is either empty or of dimension zero. We may assume that $W$ is one-dimensional. The boundary of a definable set is of a dimension smaller than that of the original definable set by \cite[Theorem 4.1.8]{vdD}. Hence, $X$ and $Y$ are both two-dimensional. We may assume that the real analytic function $f$ is identically $0$.

We show that $X$ and $Y$ satisfy the conditions of Lemma \ref{lem1}. Set $g(x)=\prod_{i=1}^l g_i(x)$; thus, it is a definable real analytic function. Consider the analytic set $Z=\{x \in \mathbb R^2\;|\; g(x)=0\}$. The restricted analytic field admits analytic decomposition by \cite[Theorem 8.8]{vdDM2}, \cite[Theorem 4.6]{DD} and \cite[Proposition 3.1.14]{Marker}. Hence there exist a finite subset $P$ of $\mathbb R^2$ and finitely many one-dimensional $\mathbb R_{\text{an}}$-definable irreducible global analytic subsets $Z_1,\ldots, Z_m$ of $\mathbb R^2$ such that $Z= P \cup \left(\bigcup_{i=1}^m Z_i\right)$, by Lemma \ref{simple3}. 

We show that $\overline{X} \cap \overline{Y} \subset Z$. Taking a point $z \in \overline{X} \cap \overline{Y}$, there exist a small positive real number $\varepsilon > 0$ and continuous definable functions $\gamma_X:[0, \varepsilon) \rightarrow \mathbb R^2$ and $\gamma_Y:[0, \varepsilon) \rightarrow \mathbb R^2$ such that $z=\gamma_X(0)=\gamma_Y(0)$, $\gamma_X(t) \in X$ and $\gamma_Y(t) \in Y$ for all $0 < t < \varepsilon$ by the curve selection lemma \cite[Corollary 6.1.5]{vdD}. The functions $g_i \circ \gamma_X$ and $g_i \circ \gamma_Y$ are continuous and positive on the open interval $(0,\varepsilon)$. Hence, we have $g_i(z) \geq 0$. If $g_i(z)$ is positive for all $i=1, \ldots,m$, then $z$ is contained in $C$. In this case, there is a continuous path in $C$ connecting $X$ with $Y$, which contradicts the assumption that $X$ is a connected component of $C$. We have shown that $g_i(z)=0$ for some $1 \leq i \leq l$. In particular, we have $\overline{X} \cap \overline{Y} \subset Z$. Since $X$ and $Y$ satisfy the conditions of Lemma \ref{lem1}, $X$ is a globally definable semianalytic set.
\end{proof}

\end{document}